\documentclass{amsart}

\usepackage{enumitem}

\usepackage{amsmath}
\usepackage{amssymb}
\usepackage{amsthm}

\newtheorem{theorem}{Theorem}[section]
\newtheorem*{theorem*}{Theorem}
\newtheorem{corollary}[theorem]{Corollary}
\newtheorem*{corollary*}{Corollary}
\newtheorem{proposition}[theorem]{Proposition}
\newtheorem*{proposition*}{Proposition}
\newtheorem{lemma}[theorem]{Lemma}
\newtheorem*{lemma*}{Lemma}

\newtheorem*{conjecture*}{Conjecture}

\newtheorem*{definition*}{Definition}
\newtheorem{remark}[theorem]{Remark}
\newtheorem*{remark*}{Remark}
\newtheorem{example}[theorem]{Example}
\newtheorem*{example*}{Example}

\newcommand{\ZZ}{\mathbb{Z}}
\newcommand{\QQ}{\mathbb{Q}}
\newcommand{\RR}{\mathbb{R}}
\newcommand{\CC}{\mathbb{C}}

\DeclareMathOperator{\reg}{reg}
\DeclareMathOperator{\disc}{disc}

\begin{document} 

\title[quartic Galois CM-fields with the same regulator]{An infinite family of pairs of \\ distinct quartic Galois CM-fields with \\ the same discriminant and regulator}
\date{}

\author[Y.~Iizuka]{Yoshichika Iizuka}
\address{Department of Mathematics, Gakushuin University, Mejiro, Toshima-ku, Tokyo 171-8588, Japan}
\email{iizuka@math.gakushuin.ac.jp}

\author[Y.~Konomi]{Yutaka Konomi}
\address{Department of Mathematics, Meijo University, Tempaku-ku, Nagoya 468-8502, Japan}
\email{konomi@meijo-u.ac.jp}

\begin{abstract}
We construct an infinite family of pairs of distinct imaginary biquadratic fields and pairs of distinct imaginary cyclic quartic fields with the same discriminant and regulator. 
We also construct an infinite family of imaginary biquadratic fields and imaginary cyclic quartic fields with the same regulator. 
Moreover, we give examples of a pair of distinct imaginary biquadratic fields and a pair of distinct imaginary cyclic quartic fields with the same discriminant, regulator and class number. 
\end{abstract}



\maketitle

\section{Introduction}

A number field is a subfield of $\mathbb{C}$ whose degree over $\mathbb{Q}$ is finite. 
Finding sufficient conditions for any two given number fields to be isomorphic or coincident is a fundamental problem. 
A famous result of such a problem is the Neukirch-Uchida theorem, which states that two number fields are isomorphic if their absolute Galois groups are isomorphic as topological groups. 
Neukirch \cite{Neukirch1969} proved the case when two number fields are Galois extensions over $\QQ$, and then Uchida \cite{Uchida1976} proved the general case.

For two number fields $K$ and $L$, to say that $K$ and $L$ are algebraically equivalent means that their Dedekind zeta functions coincide, which is equivalent to each prime $p\in\ZZ$ having the same splitting type in $K$ and $L$ (for the definition of splitting type, see Perlis \cite{Perlis1977}, \S1). 
The proposition that two algebraically equivalent number fields are isomorphic holds if they are Galois extensions over $\QQ$, but there is a counterexample in general (see Gassmann \cite{Gassmann1926}, Komatsu \cite{Komatsu1978}, Perlis \cite{Perlis1977}, \cite{Perlis1978}). 

We investigate the problem described at the beginning from the viewpoint of whether two number fields coincide if one or more of their invariants coincide. 
Specifically, we focus on the discriminant and the regulator because they are invariants of a number field appearing in the class number formula. 

Throughout this paper, for a number field $K$, we denote by $\disc(K)$, $E(K)$, $\reg(K)$, $h(K)$ and $W(K)$ the discriminant, unit group, regulator, class number and the group of roots of unity contained in $K$, respectively. 

We recall the class number formula, which relates to the residue of the Dedekind zeta function of a number field at $s=1$:

\begin{theorem}[Class Number Formula] \label{th:001_00010}
The Dedekind zeta function $\zeta_{K}(s)$ of a number field $K$ converges absolutely for $\mathrm{Re}(s) > 1$ and extends to a meromorphic function on the entire complex plane $\CC$, possessing only one simple pole at $s=1$ with residue 
\[
  \lim_{s \to 1} (s - 1) \zeta_{K}(s) = \frac{2^{r_{1}} (2\pi)^{r_{2}} h(K) \reg(K)}{\lvert W(K)\rvert \sqrt{\lvert \disc(K)\rvert}}, 
\]
where $r_{1}$ is the number of real embeddings and $2r_{2}$ is the number of complex embeddings of $K$. 
\end{theorem}

Quadratic fields coincide if their discriminants coincide; real quadratic fields coincide if their regulators coincide. 
It follows directly from the definition of discriminant that all conjugate fields of a number field have the same discriminant. 
However, neither the discriminant nor the regulator is sufficient to identify each number field of degree greater than $2$. 
Of course, it is a trivial counterexample that the regulators of imaginary quadratic fields are all $1$. 
In this paper, we give non-trivial counterexamples for imaginary biquadratic fields and for imaginary cyclic quartic fields. 
The main results are the following:

\begin{theorem} \label{th:001_00020}
For any positive real number $M$, there is a real number $R$ greater than $M$ such that the following conditions hold:
\begin{enumerate}[label=\rm(\roman*)]
\item There is an infinite family of imaginary biquadratic fields whose regulator is $R$. 
\item There is an infinite family of pairs of distinct imaginary biquadratic fields $K_{1}$ and $K_{2}$ such that $\disc(K_{1})=\disc(K_{2})$ and $\reg(K_{1})=\reg(K_{2})=R$. 
\end{enumerate}
\end{theorem}

\begin{theorem} \label{th:001_00030}
For any positive real number $M$, there is a real number $R$ greater than $M$ such that the following conditions hold:
\begin{enumerate}[label=\rm(\roman*)]
\item There is an infinite family of imaginary cyclic quartic fields whose regulator is $R$. 
\item There is an infinite family of pairs of distinct imaginary cyclic quartic fields $K_{1}$ and $K_{2}$ such that $\disc(K_{1})=\disc(K_{2})$ and $\reg(K_{1})=\reg(K_{2})=R$. 
\end{enumerate}
\end{theorem}

\begin{remark} \label{th:001_00040}
By the Hermite-Minkowski theorem, for any integer $D$, there are only finitely many number fields with discriminant $D$; 
hence the discriminant analogues of Theorems~\ref{th:001_00020}~(i) and~\ref{th:001_00030}~(i) cannot hold: there is no infinite family of pairwise distinct imaginary biquadratic fields or imaginary cyclic quartic fields having the same discriminant, and therefore in Theorems~\ref{th:001_00020}~(ii) and~\ref{th:001_00030}~(ii), the common discriminants of each pair are necessarily unbounded. 
\end{remark}

The class number, the order of the ideal class group, is one of the most important invariants of a number field. 
For a number field $K$, it is well-known that the class number of $K$ is $1$ if and only if the uniqueness of the factorization of an integer of $K$ into primes holds, and one can consider the class number as a measure of the failure to achieve the uniqueness. 
Number fields with the same class number are not necessarily isomorphic. 
However, there are some studies about the characterization of number fields that have a given class number (see Carlitz \cite{Carlitz1960}, Rush \cite{Rush1983}, Krause \cite{Krause1984}, Chapman and Smith \cite{ChapmanSmith1990}).

In Section \ref{sec:007_00010}, we give examples of a pair of distinct imaginary biquadratic fields and a pair of distinct imaginary cyclic quartic fields with the same discriminant, regulator and class number. 
This shows, in particular, that it is not possible to identify each number field in general by the residue of the Dedekind zeta function at $s=1$.

\section{Regulators of CM-fields}

Let $K$ be a CM-field and $K^{+}$ its maximal real subfield: $K^{+}=K\cap\RR$. 
Let $n$ denote the degree of the extension $K$ over $\QQ$ and $r$ the rank of the free part of $E(K)$. 
Since $K$ is totally imaginary, we have
\[
  r = \frac{n}{2} - 1. 
\]
We denote by $Q$ the index of $W(K)E(K^{+})$ in $E(K)$. 
By Washington \cite{Washington1996}, Proposition 4.16, for the ratio of the regulator of a CM-field and that of its maximal real subfield, we have
\[
  Q\cdot\reg(K) = 2^{r}\cdot\reg(K^{+}). 
\]
Moreover, by Washington \cite{Washington1996}, Theorem 4.12, we have $Q=1$ or $2$.
According to Remak \cite{Remak1954}, \S3, given a fixed maximal real subfield $K^{+}$, there are only finitely many CM-fields $K$ with $Q=2$.

The following provides a sufficient condition for $Q=1$:

\begin{theorem}[Nakamula \cite{Nakamula1996}, \S2.4, Lemma 1] \label{th:002_00010}
If $\disc(K)/\disc(K^{+})^2$ does not divide $2^{n}$, then $Q=1$.
\end{theorem}

\section{Regulators of imaginary biquadratic fields} \label{sec:003_00010}

For a square-free integer $m$, let $F_{m} = \QQ(\sqrt{m})$. 
Let $m_{1}$ and $m_{2}$ be square-free positive integers greater than $1$ such that $\gcd(m_{1}, m_{2})=1$ and $m_{1}\equiv m_{2}\equiv 1\pmod{4}$.

Since $2m_{2}\equiv 2\pmod{4}$, we have $\disc(F_{2m_{2}})=8m_{2}$. 

Since $-m_{1}\equiv 3\pmod{4}$, we have $\disc(F_{-m_{1}})=-4m_{1}$. 

Since $-2m_{1}m_{2}\equiv 2\pmod{4}$, we have $\disc(F_{-2m_{1}m_{2}})=-8m_{1}m_{2}$. 

Since $-2m_{1}\equiv 2\pmod{4}$, we have $\disc(F_{-2m_{1}})=-8m_{1}$. 

Since $-m_{1}m_{2}\equiv 3\pmod{4}$, we have $\disc(F_{-m_{1}m_{2}})=-4m_{1}m_{2}$. 

Now, for distinct square-free integers $a$ and $b$, let $B_{a,b} = \QQ(\sqrt{a},\sqrt{b})$.
The three quadratic subfields of $B_{-m_{1},2m_{2}}$ are $F_{2m_{2}}$, $F_{-m_{1}}$ and $F_{-2m_{1}m_{2}}$. Then we observe that
\[
  \disc(B_{-m_{1},2m_{2}})=\disc(F_{2m_{2}})\cdot\disc(F_{-m_{1}})\cdot\disc(F_{-2m_{1}m_{2}})=2^8m_{1}^2m_{2}^2.
\]
The three quadratic subfields of $B_{-2m_{1},2m_{2}}$ are $F_{2m_{2}}$, $F_{-2m_{1}}$ and $F_{-m_{1}m_{2}}$. Then we observe that
\[
  \disc(B_{-2m_{1},2m_{2}})=\disc(F_{2m_{2}})\cdot\disc(F_{-2m_{1}})\cdot\disc(F_{-m_{1}m_{2}})=2^8m_{1}^2m_{2}^2.
\]
Therefore, 
\[
  \disc(B_{-m_{1},2m_{2}})=\disc(B_{-2m_{1},2m_{2}}). 
\]
On the other hand, it is clear that
\[
  B_{-m_{1},2m_{2}}\neq B_{-2m_{1},2m_{2}}. 
\]
Since $F_{2m_{2}}$ is the maximal real subfield of both $B_{-m_{1},2m_{2}}$ and $B_{-2m_{1},2m_{2}}$, we obtain
\[
  \frac{\disc(B_{-m_{1},2m_{2}})}{\disc(F_{2m_{2}})^{2}} = \frac{\disc(B_{-2m_{1},2m_{2}})}{\disc(F_{2m_{2}})^{2}} = 2^{2}m_{1}^{2}. 
\]
Since $2^{2}m_{1}^{2}$ does not divide $2^{4}$, by Theorem \ref{th:002_00010}, we obtain
\[
  \reg(B_{-m_{1},2m_{2}}) = \reg(B_{-2m_{1},2m_{2}}) = 2\cdot\reg(F_{2m_{2}}). 
\]

\section{Defining polynomials of cyclic quartic fields}

For rational numbers $s$ and $t$ that are non-zero, we consider the quartic polynomial of the form
\[
  f(s, t, X) = X^{4}-2s(t^{2}+1)X^{2}+s^{2}t^{2}(t^{2}+1). 
\]
The $4$ roots of $f(s, t, X)$ in $\CC$ are
\[
  \pm\sqrt{s(t^{2}+1)\pm s\sqrt{t^{2}+1}}. 
\]
We denote by $K_{s,t}$ the splitting field of $f(s, t, X)$ over $\QQ$. Put
\begin{align*}
  \theta  &= \sqrt{s(t^{2}+1)+s\sqrt{t^{2}+1}}, \\
  \theta' &= \sqrt{s(t^{2}+1)-s\sqrt{t^{2}+1}}. 
\end{align*}
Since
\begin{align*}
  &\sqrt{t^{2} + 1} = \frac{\theta^{2}}{s} - (t^{2} + 1), \\
  &\theta^{2}\theta'^{2} = s^{2}t^{2}(t^{2}+1), 
\end{align*}
we have
\[
  \theta' = \pm\frac{st\sqrt{t^{2}+1}}{\theta}\in\QQ(\theta). 
\]
Therefore, all roots of $f(s,t,X)$ lie in $\QQ(\theta)$, and we obtain
\[
  K_{s,t} = \QQ(\theta) = \QQ\left(\!\sqrt{s(t^{2}+1)+s\sqrt{t^{2}+1}}\,\right). 
\]
Moreover, it is totally real when $s>0$ and totally imaginary when $s<0$. 

For defining polynomials of cyclic quartic fields, the following result is known: 

\begin{theorem} \label{th:004_00010}
Let $f(X)=X^{4}+aX^{2}+b\in \QQ[X]$ and suppose that $f(X)$ is irreducible over $\QQ$. 
If $b\not\in(\QQ^{\times})^{2}$ and $b(a^{2}-4b)\in(\QQ^{\times})^{2}$, 
the field generated over $\QQ$ by a root of $f(X)$ is a cyclic quartic field.
\end{theorem}

\begin{proof}
For a proof, see Jensen, Ledet and Yui \cite{JensenLedetYui2002}, Corollary 2.2.4.
\end{proof}

\begin{lemma} \label{th:004_00020}
Let $s$ and $t$ be non-zero rational numbers with $t^{2}+1\not\in\QQ^{2}$. 
Then $f(s, t, X)$ is irreducible over $\QQ$ and $K_{s,t}$ is a cyclic quartic field. 
Moreover, the real quadratic field $\QQ(\sqrt{t^{2}+1})$ is a subfield of $K_{s,t}$. 
\end{lemma}

\begin{proof}
To show that $K_{s,t}$ is a cyclic quartic field, we verify the conditions of Theorem \ref{th:004_00010} for $f(s, t, X)$.
Put
\[
  f(s,t,X)=X^{4}+aX^{2}+b
\]
with
\[
  a=-2s(t^{2}+1),\quad b=s^{2}t^{2}(t^{2}+1).
\]
Since $b=(st)^{2}(t^{2}+1)$ and $t^{2}+1\not\in\QQ^{2}$, we have $b\not\in(\QQ^{\times})^{2}$.
Moreover,
\[
  a^{2}-4b
  =4s^{2}(t^{2}+1)^{2}-4s^{2}t^{2}(t^{2}+1)
  =4s^{2}(t^{2}+1),
\]
and hence $a^{2}-4b\not\in(\QQ^{\times})^{2}$.

We prove that $f(s,t,X)$ is irreducible over $\QQ$.
Since $f(s,t,X)$ is an even polynomial, any factorization over $\QQ$ into two quadratic factors must be of one of the following forms:
\begin{align*}
  & (X^{2}+u)(X^{2}+v),\quad u,\,v\in\QQ, \\
  & (X^{2}+pX+q)(X^{2}-pX+q), \quad p,\,q\in\QQ. 
\end{align*}
In the first case, we have $a=u+v$ and $b=uv$, and hence $a^{2}-4b=(u-v)^{2}\in(\QQ^{\times})^{2}$, which contradicts $a^{2}-4b\not\in(\QQ^{\times})^{2}$. 
In the second case, the constant term is $q^{2}$, and hence $b=q^{2}\in(\QQ^{\times})^{2}$, which contradicts $b\not\in(\QQ^{\times})^{2}$.
Therefore, $f(s,t,X)$ is irreducible over $\QQ$.

Next, we compute
\begin{align*}
  b(a^{2}-4b) 
  &= s^{2}t^{2}(t^{2}+1)\cdot 4s^{2}(t^{2}+1) \\
  &= 4s^{4}t^{2}(t^{2}+1)^{2} \\
  &= \bigl(2s^{2}t(t^{2}+1)\bigr)^{2}\in(\QQ^{\times})^{2}.
\end{align*}
Thus $f(s,t,X)$ satisfies the assumption of Theorem \ref{th:004_00010}, and hence the field generated by a root of $f(s,t,X)$ is a cyclic quartic field.
In particular, if we put
\[
  \theta = \sqrt{s(t^{2}+1)+s\sqrt{t^{2}+1}},
\]
then $K_{s,t}=\QQ(\theta)$ (as shown above in this section), and therefore $K_{s,t}$ is a cyclic quartic field.

Finally, we have
\[
  \sqrt{t^{2}+1}=\frac{\theta^{2}}{s}-(t^{2}+1)\in \QQ(\theta)=K_{s,t},
\]
and hence $\QQ(\sqrt{t^{2}+1})\subseteq K_{s,t}$.
Since $t^{2}+1\not\in\QQ^{2}$, the field $\QQ(\sqrt{t^{2}+1})$ is a real quadratic field. 
\end{proof}

\begin{lemma} \label{th:004_00030}
For an integer $t$, if $t^{2}+1$ is a square, then $t=0$.
\end{lemma}

\begin{proof}
Suppose that $t^{2}+1$ is a square.
There is an integer $u$ such that $t^{2}+1 = u^{2}$. Then,
\[
  (u-t)(u+t) = u^{2}-t^{2} = 1.
\]
Since both $u-t$ and $u+t$ are integers, we obtain
\[
  (u-t, u+t) = (1, 1)\;\:\text{or}\;\:(-1, -1).
\]
In either case, we have $u-t = u+t$, which implies $t=0$. 
\end{proof}

\begin{proposition} \label{th:004_00040}
Let $s$ and $t$ be non-zero integers. Then $f(s, t, X)$ is irreducible over $\QQ$ and $K_{s,t}$ is a cyclic quartic field.
\end{proposition}

\begin{proof}
By Lemmas \ref{th:004_00020} and \ref{th:004_00030}, $f(s, t, X)$ is irreducible over $\QQ$, and hence $K_{s,t}$ is a cyclic quartic field.
\end{proof}

\begin{lemma} \label{th:004_00050}
Let $s$, $s'$ and $t$ be non-zero rational numbers satisfying $ss'\not\in\QQ^{2}$ and $t^2 + 1\not\in\QQ^{2}$. 
Then $K_{s,t}=K_{s',t}$ if and only if $t^{2}+1\in ss'\QQ^{2}$. 
\end{lemma}

\begin{proof}
We first note that
\begin{align*}
  K_{s,t} &= \QQ\left(\!\sqrt{s(t^{2}+1)+s\sqrt{t^{2}+1}}\,\right), \\
  K_{s',t} &= \QQ\left(\!\sqrt{s'(t^{2}+1)+s'\sqrt{t^{2}+1}}\,\right). 
\end{align*}

Suppose that $K_{s,t}=K_{s',t}$. Then, we have $\sqrt{s'(t^{2}+1)+s'\sqrt{t^{2}+1}}\in K_{s,t}$. Thus, 
\[
   \frac{\sqrt{ss'}}{s} = \frac{\sqrt{s'}}{\sqrt{s}}
   = \frac{\sqrt{s'(t^{2}+1)+s'\sqrt{t^{2}+1}}}{\sqrt{s(t^{2}+1)+s\sqrt{t^{2}+1}}}\in K_{s,t}.
\]
Hence, $\QQ(\sqrt{ss'})$ is a subfield of $K_{s,t}$. 
From the assumption $ss'\not\in\QQ^{2}$ of the lemma, $\QQ(\sqrt{ss'})$ is a quadratic field. 
By the assumption $t^2 + 1\not\in\QQ^{2}$ and Lemma \ref{th:004_00020}, $K_{s,t}$ contains $\QQ(\sqrt{t^{2}+1})$. 
Since $K_{s,t}$ is cyclic, it has only one quadratic subfield. 
Therefore, we obtain $\QQ(\sqrt{ss'})=\QQ(\sqrt{t^{2}+1})$, which implies that $t^{2}+1\in ss'\QQ^{2}$. 

Conversely, suppose that $t^{2}+1\in ss'\QQ^{2}$. 
Since $t^2 + 1\not\in\QQ^{2}$, by Lemma \ref{th:004_00020}, 
we have $\sqrt{t^2 + 1}\in K_{s,t}$, and hence $\sqrt{ss'}\in K_{s,t}$. Then we observe that
\[
  \sqrt{\frac{s'}{s}}=\frac{\sqrt{ss'}}{s}\in K_{s,t}.
\]
It follows that
\[
  \sqrt{s'(t^{2}+1)+s'\sqrt{t^{2}+1}}
  = \sqrt{\frac{s'}{s}}\cdot\sqrt{s(t^{2}+1)+s\sqrt{t^{2}+1}}\in K_{s,t}, 
\]
which yields $K_{s',t}\subseteq K_{s,t}$. We can show the reverse inclusion in the same way. 
\end{proof}

\begin{remark} \label{th:004_00060}
For the case $s=s'$, it is clear that $K_{s,t}=K_{s',t}$ holds, but $ss'\not\in\QQ^{2}$ does not hold. 
\end{remark}

Now, by Proposition \ref{th:004_00040}, $K_{s,t}$ and $K_{2s,t}$ are cyclic quartic fields.
By Lemma \ref{th:004_00020}, both contain the quadratic field $\QQ(\sqrt{t^{2}+1})$.

\begin{lemma} \label{th:004_00070}
Let $t$ be an integer with $t\equiv 3$ or $5\pmod{8}$. Then, $t^{2}+1\not\in 2\QQ^{2}$.
\end{lemma}

\begin{proof}
Let $t$ be an integer with $t\equiv 3$ or $5\pmod{8}$. 
Suppose that $t^{2}+1\in 2\QQ^{2}$.
Then there exists $q\in\QQ$ such that $t^{2}+1=2q^{2}$. 
Moreover, $q$ can be expressed as an irreducible fraction as follows:
\[
q=\frac{m}{n},\quad \gcd(m,n)=1,\quad n>0.
\]
Then we have
\[
t^{2}+1=2\cdot\left(\frac{m}{n}\right)^{2}.
\]
Multiplying both sides by $n^{2}$ and dividing by $2$, we obtain
\[
\frac{t^{2}+1}{2}\cdot n^{2}=m^{2}.
\]
Since $t$ is odd, $(t^{2}+1)/2$ on the left-hand side is an integer. 
If $n$ has a prime factor $p$, then $p^{2}\mid m^{2}$, implying $p\mid m$, which contradicts $\gcd(m,n)=1$. 
Hence $n=1$. Therefore,
\[
\frac{t^{2}+1}{2} = m^{2}.
\]
On the other hand, since $t\equiv 3$ or $5\pmod{8}$, we have $t^{2}\equiv 9\pmod{16}$. Thus,
\[
\frac{t^{2}+1}{2}\equiv \frac{10}{2}\equiv 5\pmod{8}.
\]
Hence,
\[
m^{2} \equiv 5\pmod{8}.
\]
However, the only possible values of a square of an integer modulo $8$ are $0$, $1$ and $4$, which is a contradiction.
Therefore, we must have $t^{2}+1\not\in 2\QQ^{2}$.
\end{proof}

\begin{proposition} \label{th:004_00080}
Let $s$ and $t$ be non-zero integers with $t\equiv 3$ or $5\pmod{8}$. Then, $K_{s,t}\neq K_{2s,t}$.
\end{proposition}

\begin{proof}
Putting $s'=2s$ gives 
\[
  ss'\QQ^{2} = 2s^{2}\QQ^{2} = 2\QQ^{2}.
\] 
By Lemma \ref{th:004_00070}, we have $t^{2}+1\not\in ss'\QQ^{2}$. 
Since $s$ and $t$ are non-zero integers, by Lemmas \ref{th:004_00030} and \ref{th:004_00050}, we obtain the desired result. 
\end{proof}

\section{Discriminants of cyclic quartic fields}

The following result concerning discriminants of cyclic quartic fields is extracted from Corollary 4 in Huard, Spearman and Williams \cite{HuardSpearmanWilliams1995}: 

\begin{theorem} \label{th:005_00010}
Let $a$, $b$ and $c$ be integers such that $c$ and $\gcd(a, b)$ are both square-free. 
Suppose that either of the following conditions (a) or (b) hold:
\begin{enumerate}[label=\rm(\alph*)]
  \item $a\equiv 4\pmod{8}$, $b\equiv 2\pmod{4}$ and $c\equiv 2\pmod{8}$. 
  \item $a\equiv 2\pmod{4}$, $b\equiv 1\pmod{2}$ and $c\equiv 2\pmod{8}$. 
\end{enumerate}
If $\QQ\left(\!\sqrt{a+b\sqrt{c}}\,\right)$ is a cyclic quartic field, then its discriminant is
\[
  2^{8}\cdot\frac{\gcd(a, b)^{2}\cdot c^{3}}{\gcd(a, b, c)^{2}}. 
\]
\end{theorem}

\begin{proposition} \label{th:005_00020}
Let $s$ and $t$ be integers such that $s$ is square-free, $t\equiv 1\pmod{2}$ and $\gcd(s, t^{2}+1)=1$. 
Then $K_{s,t}$, $K_{2s,t}$, $K_{-s,t}$ and $K_{-2s,t}$ have the same discriminant. 
Furthermore, the following conditions hold: 
\begin{enumerate}[label=\rm(\roman*)]
 \item The $2$-exponent of the discriminant is $11$.
 \item For any odd prime $p$, if $p\mid s$, then the $p$-exponent of the discriminant is $2$.
 \item For any odd prime $q$, if $q\nmid s$ and $q\nmid t^{2}+1$, then the $q$-exponent of the discriminant is $0$.
\end{enumerate}
\end{proposition}

\begin{proof}
We will prove only for $K_{s,t}$ and $K_{2s,t}$, but it is exactly the same for $K_{-s,t}$ and $K_{-2s,t}$. We first obtain 
\begin{align*}
 K_{s,t} &= \QQ\left(\!\sqrt{s(t^{2}+1)+s\sqrt{t^{2}+1}}\,\right), \\
 K_{2s,t} &= \QQ\left(\!\sqrt{2s(t^{2}+1)+2s\sqrt{t^{2}+1}}\,\right).
\end{align*}
There are positive integers $y$ and $m$ such that
\[
 t^{2}+1 = y^{2}m,\quad \text{where $m$ is square-free}. 
\]
Also, there are positive integers $x$ and $n$ such that
\[
 y = x^{2}n,\quad \text{where $n$ is square-free}. 
\]
Thus,
\[
 t^{2}+1=x^{4}n^{2}m. 
\]
Therefore,
\begin{align*}
  & \sqrt{s(t^{2}+1)+s\sqrt{t^{2}+1}} = x\sqrt{sx^{2}n^{2}m + sn\sqrt{m}}, \\
  & \sqrt{2s(t^{2}+1)+2s\sqrt{t^{2}+1}} = x\sqrt{2sx^{2}n^{2}m + 2sn\sqrt{m}}.
\end{align*}
That is, 
\begin{align*}
 K_{s,t} &= \QQ\left(\!\sqrt{sx^{2}n^{2}m + sn\sqrt{m}}\,\right), \\
 K_{2s,t} &= \QQ\left(\!\sqrt{2sx^{2}n^{2}m + 2sn\sqrt{m}}\,\right).
\end{align*}
Since $t\equiv 1\pmod{2}$, we have $t^{2}+1\equiv 2\pmod{8}$, which implies that the $2$-exponent of $t^{2}+1$ is $1$. Hence,
\[
 x^{4}n^{2} = y^{2}\equiv 1\pmod{8},\quad m\equiv 2\pmod{8}.
\]
Therefore, $x$ and $n$ are odd, and $m$ is even. 

Putting $a=sx^{2}n^{2}m$, $b=sn$ and $c=m$ gives
\begin{align*}
 &\gcd(a, b) = sn, \\   
 &\gcd(a, b, c) = \gcd(\gcd(a, b), c) = \gcd(m, sn).
\end{align*}
By the assumption on $s$, we see that $\gcd(a, b)$ is square-free and $\gcd(m, sn)=\gcd(m, n)$. 
By Theorem \ref{th:005_00010}, the discriminant of $K_{s,t}$ is $2^{8}s^{2}n^{2}m^{3}/\gcd(m,n)^{2}$.

Putting $a=2sx^{2}n^{2}m$, $b=2sn$ and $c=m$ gives
\begin{align*}
 &\gcd(a, b) = 2sn, \\   
 &\gcd(a, b, c) = \gcd(\gcd(a, b), c) = \gcd(m, 2sn).
\end{align*}
By the assumption on $s$, we see that $\gcd(a, b)$ is square-free.
Since the $2$-exponent of $m$ is $1$, and $s$ and $n$ are odd, we observe that
\[
 \gcd(m, 2sn) = 2\cdot\gcd(m/2, sn) = 2\cdot\gcd(m, sn) = 2\cdot\gcd(m, n).
\]
By Theorem \ref{th:005_00010}, the discriminant of $K_{2s,t}$ is also $2^{8}s^{2}n^{2}m^{3}/\gcd(m,n)^{2}$.

Conditions (i), (ii) and (iii) are clear from the explicit form of the discriminant. 
In particular, for condition (iii), we note that if a prime $p$ divides $n^{2}m^{3}/\gcd(m,n)^{2}$, then $p$ divides $t^{2}+1$. 
\end{proof}

\begin{corollary} \label{th:005_00030}
There are infinitely many pairs of different real cyclic quartic fields with the same discriminant. The same holds for imaginary fields.
\end{corollary}

\begin{proof}
Fix an integer $t$ satisfying $t\equiv 3$ or $t\equiv 5\pmod{8}$.
Let $p$ be an odd prime greater than $t^{2}+1$.
The quartic fields $K_{p,t}$ and $K_{2p,t}$ are both real by definition, and distinct by Proposition \ref{th:004_00080}. 
Since $p$ is square-free and $\gcd(p, t^{2}+1)=1$, it follows from Proposition \ref{th:005_00020} that $K_{p,t}$ and $K_{2p,t}$ have the same discriminant. 
For each $p$, the pair $(K_{p,t}, K_{2p,t})$ is distinct because the discriminant of $K_{p,t}$ is distinct. 

We can prove the same for pairs of imaginary cyclic quartic fields $K_{-p,t}$ and $K_{-2p,t}$. 
\end{proof}

\begin{remark}  \label{th:005_00040}
Fix an integer $t$ satisfying $t\equiv 1\pmod{2}$ and consider $K_{p,t}$ and $K_{-p,t}$ for odd primes $p$ greater than $t^{2}+1$. 
Then, it is clear that there are infinitely many pairs of real and imaginary cyclic quartic fields with the same discriminant.
\end{remark}

\section{The regulator of the real quadratic field $\QQ(\sqrt{t^{2}+1})$}

The following result is due to Nagell \cite{Nagell1922}: 

\begin{theorem} \label{th:006_00010}
Let $g(X)\in\ZZ[X]$ be a polynomial satisfying the following conditions:
\begin{enumerate}[label=\rm(C\arabic*)]
\item The degree of any irreducible polynomial dividing $g(X)$ is at most $2$. 
\item $g(X)$ has no multiple roots. 
\item $g(X)$ is primitive. 
\item For any prime number $p$, there is a positive integer $k_0$ such that $p^{2}\nmid g(k_0)$. 
\end{enumerate}
Then there are infinitely many positive integers $k$ such that $g(k)$ is square-free.
\end{theorem}

\begin{lemma} \label{th:006_00020}
There are infinitely many positive integers $t$ such that $t^{2}+1$ is square-free and $t\equiv 5\pmod{8}$.
\end{lemma}

\begin{proof}
Let $f(X) = (8X+5)^2 + 1$. 
It suffices to show that there are infinitely many positive integers $k$ such that $f(k)$ is square-free. 
Expanding the right-hand side of the defining equation of $f(X)$, we have
\[
  f(X) = 64X^2 + 80X + 26 = 2(32X^2 + 40X + 13). 
\]
Let $g(X)=32X^2 + 40X + 13$. 
It is clear that $g(X)$ satisfies conditions (C1), (C2), and (C3) of Theorem \ref{th:006_00010}. 
Since $g(1)=85=5\cdot 17$, condition (C4) is also satisfied. Therefore, there are infinitely many positive integers $k$ such that $g(k)$ is square-free. 
On the other hand, for each integer $k$, since $g(k)$ is odd, if $g(k)$ is square-free, then $f(k)=2g(k)$ is also square-free. 
Hence, there are infinitely many positive integers $k$ such that $f(k)$ is square-free. 
\end{proof}

\begin{lemma} \label{th:006_00030}
Let $t$ be an odd integer such that $t^{2}+1$ is square-free. 
Then $t+\sqrt{t^{2}+1}$ is a fundamental unit of the quadratic field $\QQ(\sqrt{t^{2}+1})$. 
\end{lemma}

This classical result can be found in Degert \cite[p.~6]{Degert1958} or Hasse \cite[p.~49]{Hasse1965};
we include a proof for the reader's convenience. 

\begin{proof}
Let $t$ be an odd integer such that $t^{2}+1$ is square-free. 
The minimal positive integer solution of the Pell equation $X^{2}-(t^{2}+1)Y^{2} = -1$ is $(X, Y)=(t, 1)$, and $t^{2}+1\equiv 2\pmod{4}$. 
Therefore, $t+\sqrt{t^{2}+1}$ is a fundamental unit of the quadratic field $\QQ(\sqrt{t^{2}+1})$. 
\end{proof}

\begin{proposition} \label{th:006_00040}
For any positive real number $M$, there is a positive integer $t$ such that 
$t\equiv 5\pmod{8}$ and the regulator of the real quadratic field $\QQ(\sqrt{t^{2}+1})$ is greater than $M$. 
\end{proposition}

\begin{proof}
Let $M$ be a positive real number. By Lemma \ref{th:006_00020}, 
there is a positive integer $t$ such that $t^{2}+1$ is square-free, $t\equiv 5\pmod{8}$, and $t>e^{M}$. 
By Lemma \ref{th:006_00030}, we have
\[
  \reg\left(\QQ(\sqrt{t^{2}+1})\right) = \log(t+\sqrt{t^{2}+1}) > \log t > M. 
\]
\end{proof}

\begin{remark} \label{th:006_00050}
An analogous result holds if we replace the condition $t \equiv 5 \pmod{8}$ with $t \equiv 3 \pmod{8}$ in Proposition \ref{th:006_00040}.
This follows by replacing Lemma \ref{th:006_00020} with its analogue for $t \equiv 3 \pmod{8}$.
\end{remark}

\section{Proofs of the main theorems and examples} \label{sec:007_00010}

\begin{proof}[Proof of Theorem \ref{th:001_00020}]
Let $M$ be a positive real number. By Proposition \ref{th:006_00040}, 
there is a positive integer $t$ such that $t\equiv 5\pmod{8}$ and the regulator $R_{0}$ of the real quadratic field $\QQ(\sqrt{t^{2}+1})$ is greater than $M$. 
Furthermore, by Dirichlet's theorem on arithmetic progressions, 
there are infinitely many primes $p$ satisfying $p>t^{2}+1$ (hence in particular, $\gcd(p, t^{2}+1)=1$) and $p\equiv 1\pmod{4}$. 
For each such prime $p$, the biquadratic fields $B_{-p,t^{2}+1}$ are distinct because different values of $p$ yield different combinations of three quadratic subfields. 
By the discussion in Section \ref{sec:003_00010}, 
these biquadratic fields have the same maximal real subfield $\QQ(\sqrt{t^{2}+1})$ and hence the same regulator $R=2R_{0}$. This proves (i).

On the other hand, again by the discussion in Section \ref{sec:003_00010}, 
$B_{-p,t^{2}+1}$ and $B_{-2p,t^{2}+1}$ are distinct biquadratic fields having the same discriminant and the same regulator $R$. This proves (ii).
\end{proof}

\begin{proof}[Proof of Theorem \ref{th:001_00030}]
Let $M$ be a positive real number. By Proposition \ref{th:006_00040}, 
there is a positive integer $t$ such that $t\equiv 5\pmod{8}$ and the regulator $R_{0}$ of the real quadratic field $\QQ(\sqrt{t^{2}+1})$ is greater than $M$. 
For each odd prime $p$ satisfying $p>t^{2}+1$, the cyclic quartic fields $K_{-p,t}$ are distinct because their discriminants are distinct by Proposition \ref{th:005_00020}. 
It follows from Lemma \ref{th:004_00020} that all such fields $K_{-p,t}$ have the same maximal real subfield $\QQ(\sqrt{t^{2}+1})$. 
Since all of them are CM-fields, by Theorem \ref{th:002_00010}, they have the same regulator $R=2R_{0}$. This proves (i).

Now, $K_{-p,t}$ and $K_{-2p,t}$ are cyclic quartic fields by Proposition \ref{th:004_00040}, and they are clearly imaginary. 
By Proposition \ref{th:004_00080}, $K_{-p,t}\neq K_{-2p,t}$. 
The prime $p$ is square-free and satisfies $\gcd(p, t^{2}+1)=1$. 
By Proposition \ref{th:005_00020}, $K_{-p,t}$ and $K_{-2p,t}$ have the same discriminant. 
By Lemma \ref{th:004_00020}, they have the same maximal real subfield. 
Since both of them are CM-fields, by Theorem \ref{th:002_00010}, they have the same regulator $R$. This proves (ii). 
\end{proof}

Finally, we give examples of pairs of distinct imaginary biquadratic fields and pairs of distinct cyclic quartic fields having the same discriminant, regulator and class number. 
For the computation of fundamental units, we use Lemma \ref{th:006_00030}. 
For the computation of class numbers, we use Magma \cite{BosmaCannonPlayoust1997}.

\begin{example} \label{th:007_00010}
We consider the pair of the biquadratic fields $B_{-21, 10}$ and $B_{-42, 10}$. 
They share the maximal real subfield $\QQ(\sqrt{10})$, whose fundamental unit is $\varepsilon = 3 + \sqrt{10}$. 
Moreover, they have the same discriminant $2^{8}\cdot 3^{2}\cdot 5^{2}\cdot 7^{2}$, regulator $2\cdot\log\lvert\varepsilon\rvert = 3.6368929\cdots$ and class number $32$. 
\end{example}

\begin{example} \label{th:007_00020}
We consider the pair of the cyclic quartic fields $K_{-3,35}$ and $K_{-6,35}$. 
They share the maximal real subfield $\QQ(\sqrt{2\cdot 613})$, whose fundamental unit is $\varepsilon = 35 + \sqrt{2\cdot 613}$. 
Moreover, they have the same discriminant $2^{11}\cdot 3^{2}\cdot 613^{3}$, regulator $2\cdot\log\lvert\varepsilon\rvert = 8.4973985\cdots$ and class number $2^{3}\cdot 5^{2}\cdot 97$. 
\end{example}

\begin{remark} \label{th:007_00030}
The class number of a CM-field is divisible by the class number of its maximal real subfield. 
\end{remark}

\begin{remark}\label{th:007_00040}
For an odd integer $t$ greater than $1$, if $t^{2}+1\not\in 2\QQ^{2}$, the discriminant of the real quadratic field $\QQ(\sqrt{t^{2}+1})$ has more than one prime factor, and hence by genus theory the narrow class number is even. Since the fundamental unit $t+\sqrt{t^{2}+1}$ of $\QQ(\sqrt{t^{2}+1})$ has norm $-1$, the narrow and usual class numbers coincide, and therefore the class number is even. 
\end{remark}

\section*{Acknowledgements}

The authors would like to thank Professor Shin Nakano for his valuable comments and the anonymous referee for careful reading and helpful suggestions.

\end{document}